\documentclass{amsart}
\usepackage{amsmath}
\usepackage{amsfonts}
\usepackage{amssymb}
\usepackage{mathrsfs}
\usepackage{epsfig}
\usepackage{psfrag}
\usepackage{verbatim}
\usepackage{xy}

\begin{document}
\title{The Schreier continuum and ends}


\author{Alex Clark
}
\address{Alex Clark, Department of Mathematics, University of Leicester, University Road, Leicester LE1 7RH, United Kingdom}
\email{adc20@le.ac.uk}

\author{Robbert Fokkink}

\address{Robbert Fokkink, DIAM Probability, TU Delft, Mekelweg 4, 2628CD Delft, Netherlands}
\email{R.J.Fokkink@tudelft.nl}

\author{Olga Lukina}
\address{Olga Lukina, Department of Mathematics, University of Leicester, University Road, Leicester LE1 7RH, United Kingdom}
\email{ollukina940@gmail.com}


\thanks{Authors supported by NWO travel grant 040.11.132}
\thanks{AC and OL supported in part by  EPSRC grant EP/G006377/1}
\maketitle

\begin{abstract}
\noindent Blanc showed in his thesis that a compact minimal foliated space with a residual subset of $2$-ended leaves can contain only $1$ or $2$ ended leaves. In this paper we give examples of compact minimal foliated spaces where a topologically generic leaf has $1$ end, there is an uncountable set of leaves with $2$ ends and a leaf with $2n$ ends, for a given $n>1$. The examples we present are weak solenoids, which allows us to represent the graph of the group action on the fibre as the inverse limit of finite coverings of a finite graph, which we call the Schreier continuum, which we use to obtain the result. While in certain cases the problem can be reduced to the study of a self-similar action of an automorphism group of a regular tree, our geometric technique is more general, as it applies to cases where the action is not self-similar.
\end{abstract}


\newtheorem{Lemma}{Lemma}[section]
\newtheorem{Prop}{Proposition}[section]
\newtheorem{Cor}{Corollary}[section]
\newtheorem{Theor}{Theorem}[section]
\newtheorem{Def}{Definition}[section]
\newtheorem{Example}{Example}[section]
\newtheorem{Remark}{Remark}[section]
\newtheorem{Term}{Terminology}[section]
\newtheorem{Q}{Question}

\newcommand{\Id}{\mathrm{Id}}

\newcommand{\G}{\Gamma}

\newcommand{\g}{\gamma}
\newcommand{\de}{\delta}
\newcommand{\image}{\mathrm{image }\ }
\newcommand{\rank}{\mathrm{rank }}
\newcommand{\corank}{\mathrm{corank }}
\newcommand{\id}{\mathrm{id}}
\newcommand{\Ad}{\mathrm{Ad}}
\newcommand{\F}{\mathcal{F}}
\newcommand{\K}{\mathcal{K}}
\newcommand{\Aut}{\mathrm{Aut}}
\newcommand{\const}{\mathrm{const}}
\newcommand{\grad}{\mathrm{grad}}
\newcommand{\col}{\mathrm{col}}
\newcommand{\supp}{\mathrm{supp}}
\newcommand{\Hom}{\mathrm{Hom}}
\newcommand{\breukje}[2]{\mbox{\small $\frac{#1}{#2}$}}
\newcommand{\mb}[1]{\mbox{$\mathbb{#1}$}}
\newcommand{\ms}[1]{\mbox{$\mathscr{#1}$}}
\newcommand{\mc}[1]{\mbox{$\mathcal{#1}$}}
\newcommand{\trm}[1]{\mbox{${\textrm{#1}}$}}

\renewcommand{\theequation}{\thesection.\arabic{equation}}

\section{Introduction}

A compact foliated space $X$ is a compact metrisable space together with a decomposition of $X$ into subsets $S_\alpha$, called \emph{leaves}, such that for each $x \in X$ there is a chart $\varphi_x:U_x \to \mb{R}^n \times Z$, where $Z$ is a compact metrisable space, and the restriction of $\varphi_x$ onto a connected component of $S_\alpha \cap U$ is constant on the second component \cite{CCbook1}. One class of examples of foliated spaces arising in continuum theory are weak solenoids; i.e., the inverse limits $X_\infty = \lim_{\longleftarrow}\{f^i_{i-1}:X_i \to X_{i-1}\}$, where $X_i$ is a closed manifold, and each bonding map $f^i_{i-1}$ is a finite-to-one covering map \cite{FO}.

Weak solenoids are natural generalisations of the well-known one-dimensional solenoids and have the structure of a foliated
bundle over a closed manifold with a Cantor set fibre. Leaves in solenoids are their path-connected components, and every leaf in a solenoid is dense. One-dimensional solenoids are important in the study of flows and
have been identified as minimal sets of smooth flows~\cite{BF,GST,GT,Kan,T} and are
known to occur generically in Hamiltonian flows on closed
manifolds~\cite{MM}. The more general solenoids introduced in~\cite{Mc} provide a class of higher-dimensional continua
that have been the object of great interest in the study of continua ~\cite{RT,S,FO}. It was shown in~\cite{CH1} that the
equicontinuous minimal sets of foliations having the local structure
of a disk in a leaf times a Cantor set are solenoids, and
in~\cite{CH} examples are provided of certain classes of solenoids
embedded as minimal sets of smooth foliations of manifolds. End structures of leaves provide the means to study their behaviour at infinity, and so are important for the study of the inner structure of solenoids.

Ghys shows in his seminal work~\cite{G} that the
leaves of almost all (with respect to an appropriate harmonic
measure) points of a compact foliated space have either $0$, $1$,
$2$ or a Cantor set of ends. The harmonic measures used are those introduced by Garnett \cite{Gar} and locally have the form of the product of a harmonic function and the Riemannian volume on the plaques of the leaves in foliation charts. These measures are used in foliation theory since, in contrast to invariant transverse measures, they always exist and at the same time many of the basic results from ergodic theory carry over to these harmonic measures.  Inspired by this
result, Cantwell and Conlon~\cite{CC1} obtained similar results in
the topological setting. They show that for a complete foliated space with a totally recurrent leaf (in particular, for compact minimal sets), the leaves of all points of a residual subset either have no
ends, one end, two ends or a Cantor set of ends. Blanc \cite{B} showed that if a compact minimal foliated space $X$ has a topologically generic leaf with $2$ ends, then any other leaf in $X$ has either $1$ or $2$ ends. In this paper we show that if a compact foliated space has a topologically generic leaf with $1$ end, then the maximal number of ends can be made an  arbitrarily large finite number. In particular, we prove the following result.

\begin{Theor}\label{Theor-one}
Given $n>1$, there exists a weak solenoid $X_n$ where the union of leaves with $1$ end forms a residual subset, the union of leaves with $2$ ends forms a meager subset and the number of such leaves is uncountable, and there is a single leaf with $2n$ ends.
\end{Theor}

For each $n>1$ in Theorem \ref{Theor-one} an example of a foliated space with required properties is obtained by a suitable modification of the Schori solenoid \cite{S}. The Schori solenoid is the inverse limit of an inverse sequence of non-regular $3$-to-$1$ coverings of a genus $2$ surface $X_0$ and, as any solenoid, is a foliated fibre bundle with the base $X_0$. An important property is that every leaf $L$ in a foliated bundle $p:E \to X_0$ is a covering space of the base, and so by the result of Scott~\cite[Lemma 1.2]{Sc} the number of ends of $L$ is the same as the number of ends of the pair of groups $(\pi_1(X_0,p(x)), p_*\pi_1(L,x))$, or the number of ends of the Schreier diagram of the coset space $\pi_1(E,p(x))/ p_*\pi_1(L,x)$. The Schreier space of $E$ is then the set $\Lambda_\infty$ of all Schreier diagrams of leaves, bound together in a suitable topology so that they preserve the essential dynamical properties of $E$. In the case when $E = X_\infty$ is a weak solenoid, there are intermediate coverings $ E \to^{p_i}  X_i \to^{p^i_0}   X_0$ which allow us to obtain a representation of $\Lambda_\infty$ as the inverse limit of finite graphs, and make computations by studying the projections of path components of $\Lambda_\infty$ onto these finite graphs.

It is interesting to compare these results with the recent results~\cite{BDN,DDMN}, where the number of ends of a large class of orbital Schreier graphs arising in the context of self-similar group actions is determined. There,  general results are obtained, showing that in a measure-theoretic sense most orbital Schreier graphs have either one or two ends, and for most self-similar groups the generic end structure is one. In the case when a weak solenoid $X_\infty$ has a representation $X_\infty = \lim_{\longleftarrow}\{ p^i_{i-1}:X_i \to X_{i-1}\}$ such that every covering map $p^i_{i-1}$ has the same degree and does not significantly differ from the other covering maps in the sequence, the image of the monodromy representation $\pi_1(X_0,p(x)) \to \textrm{Homeo} (F)$, where $F \cong p^{-1}(p(x))$ is a fibre of the weak solenoid, can be identified with a group of automorphisms of a regular tree. Such an action is self-similar, and is generated by a finite automaton. However, the class of weak solenoids is a lot richer than those with self-similar action of the monodromy group, and our geometric approach is applicable to any of them. In Section \ref{subsec-nonselfsimilar} we compute end structures for a weak solenoid which is the inverse limit of a sequence of covering maps of variable degree, where a map of each degree occurs an infinite number of times. As is well-known, for solenoids of dimension $1$, i.e. inverse limits of finite coverings of $\mb{S}^1$, equivalence classes of sequences of degrees provide a classification up to a homeomorphism, and for dimension higher than $1$ determining whether a weak solenoid has a representation as the inverse limit of coverings of constant degree is by no means a trivial question. Our geometric approach  overcomes this difficulty.

The results of \cite{BDN,DDMN} reveal connections of our work to the study of Julia sets, and leads to interesting questions concerning the interplay between topological and measure-theoretical genericity of various properties of foliated manifolds, such as those in Section \ref{sec:conclusions}.

In detail, our paper is organised as follows. In Section
\ref{sec:sc} we recall some background knowledge from algebraic topology and foliation theory, and give a brief outline of the computation technique of the Schreier continuum. Section
\ref{sec:schoriexample} contains computations of the end structures
for some specific solenoids.
Section
\ref{sec:conclusions} indicates some open questions.

\section{Preliminaries: ends of covering spaces and foliated bundles}\label{sec:sc}

The set of ends of a non-compact separable metric topological space $X$ is a set of ideal points at infinity that compactify $X$. Ends appear first in the work of Freudenthal \cite{F}, and the compactification by ends is often called the \emph{Freudental compactification}. More precisely \cite{CCbook1}, given a sequence of compact subsets
 $$K_{1} \subset K_{2} \subset K_{3} \subset \cdots \phantom{--} {\rm with} \phantom{--} \bigcup_i K_{i} = X$$
consider chains of unbounded connected components of $X \backslash K_{i}$
   $$\{U_{i}\}: U_{1} \supset U_{2} \supset U_{3} \supset \ldots  \phantom{--} {\rm where} \phantom{--} \bigcap_i U_{i} = \emptyset.$$
Set $\{U_{i}\} \sim \{V_{j}\}$ if for any $i$ there is a $j>i$ such that
  $$(U_{i}\cap V_{i}) \supset (U_{j} \cup V_{j}).$$
Then an \emph{end} $e$ of $X$ is an equivalence class of chains $\{U_{i}\} $. Denote by $\mathcal{E}(X)$ the set of ends, and set $X^* = X \cup \mathcal{E}(X)$. Put a topology on $X^*$ by saying that sets open in $X$ are open in $X^*$, and if $\{U_{i}\}$ is a sequence representing an end $e$, then $U_{i}$ together with all ends contained in $U_{i}$ is open. With this topology, $X^*$ is compact, $X$ is open in $X^*$ and $\mathcal{E}(X)$ is a totally disconnected subset of $X^*$ \cite{CCbook1}.

\subsection{Ends of covering spaces}

Recall (see, e.g. \cite{Bow}) that given a group $G$ generated by
$S=\{s_1,\dots,s_n\}$, the \emph{Cayley graph} $\G$ is the graph
whose vertices are the elements of $G$ and for which there is an
edge labelled by $s_i \in S$ joining $g$ to $g'$ if and only if
$g'=g\,s_i$. Then $G$ acts on the left of $\G$ in a natural way.
Given a subgroup $C$ of $G$ the \emph{Schreier diagram} of $(G,C)$
is the orbit space $\G/C$ of the left action of the subgroup $C$ on
$\G$. While the actual construction of this graph depends on the
choice of generators $S$, the number of ends of the graph is
independent of this choice (see, e.g. \cite{Bow}).

In his classic work Hopf~\cite{Ho} showed that a non-compact regular covering space of a compact polyhedron has
either one (e.g. the plane), two (e.g. the line) or a Cantor set of
ends (e.g. the universal cover of the figure eight). Here a covering space $p:L \to B$ is regular if $p_*\pi_1 (L, x)$ is a normal subgroup of $\pi_1(B,p(x))$, and geometrically the number of ends of $L$ coincides with that of the Cayley diagram of the quotient group $\pi_1(B,p(x))/p_*\pi_1 (L, x)$. In the case when $p: L\to B$ is non-regular, the generalisation of Hopf's theorem to irregular coverings (see, e.g., Scott~\cite[Lemma 1.1 - 1.2]{Sc}) yields that the number of ends of $L$ is the same as the number of ends of the Schreier diagram of the coset space $\pi_1(B,p(x))/p_*\pi_1 (L, x)$.

\subsection{Ends of leaves in foliated bundles}

Let $E$ be a compact metrisable space with foliation $\mathcal{F}$, that is \cite{CCbook1}, there is a decomposition of $E$ into disjiont subsets $\{L_\alpha\}_{\alpha \in A}$ called leaves of the foliation $\mathcal{F}$, such that each $x \in E$ has an open neighborhood $U_x$ with a homeomorphism $\varphi_x:U_x \to V_x \times Z_x$, where $V_x \subset \mathbb{R}^n$ and $Z_x$ is a compact metrisable space, and connected components of $L_\alpha \cap U_x$ are given by fixing values in $Z_x$. The sets $\varphi_x^{-1}(V_x \times \{z\})$, $z \in Z_x$, are open sets of the \emph{leaf topology} on $L_\alpha$. If the transition maps $\varphi_x \circ \varphi^{-1}_y$ are differentiable with respect to coordinates on $\mathbb{R}^n$, then the leaves $L_\alpha$ are smooth manifolds, and $\mathcal{F}$ is a smooth foliation.

Suppose there is a locally trivial projection $p:E \rightarrow B$ on a closed manifold $B$, such that leaves of $\F$ are transverse to the fibres of $p$. Then $p:E \to B$ is a \emph{foliated bundle} (see, e.g., Candel and Conlon \cite[Example 2.1.5]{CCbook1}), and, in particular, for each leaf $L \subset E$ the restriction $p|_L: L \to B$ in the leaf topology is a covering space.
Given $b \in B$, there is a \emph{total holonomy} homomorphism
\[
h : \pi_1(B,b) \rightarrow \mathrm{Homeo}(F)
\]
determined by the lifting of loops in $B$ based at $b$ to paths
contained in the leaves of $\F$. This allows one to translate many
questions about $\F$ into questions involving subgroups of
$\mathrm{Homeo}(F)$. Associated to each point $x \in F$ is the
subgroup of $\pi_1(B,b)$ corresponding to the stabiliser of $x$ for
the induced action on $F$, which we will refer to as the
\emph{kernel} of $x$ and denote $\K_x$.

In particular, if $L_x$ is a leaf containing $x$, then by standard algebraic topological arguments (see, e.g., \cite{Massey}) we have $\K_x = p_* \pi_1(L_x,x)$, and the leaf $L_x$ is homeomorphic to the quotient $\widetilde{B}/\K_x$, where $\widetilde{B}$ is the universal cover of $B$. It is clear that if $y \in L_x$, then $\K_x = \K_y$.

When the foliated bundle $p:E \to B$ is principal (that is, when the bundle automorphisms act transitively on the fibres of $p$), then all the leaves of $\F$ are regular coverings of $B$ and so Hopf's theorem applies. Moreover, in this case the bundle automorphisms yield a homeomorphism between any two leaves and so all leaves have homeomorphic end spaces, either $0$ (compact), $1$, $2$ or a Cantor set. If $p:E \to B$ is not principal, the kernel $\K_x$, $x \in E$, may vary depending on the point.

\subsection{The Schreier space of a foliated bundle}\label{subsec:schreierspace}
We construct the \emph{Schreier space} of a foliated bundle $p:E \to B$, where $B$ is a closed manifold. We fix a finite set of generators $S=\{[\g_1],\dots,[\g_n]\}$ of  $\pi_1(B,b)$ that are represented by loops $\g_i$ in $B$ that intersect pairwise only at the base point $b$. If the dimension of $B$ is $2$ we can choose the generators to be given in the standard way by the boundary edges of a polygon whose quotient is homeomorphic to $B$, and when the dimension of $B$ is greater than $2$ it is not hard to see how to construct these loops given any finite set of generators and local charts.

\begin{Def}
For a given foliated bundle $p:E \to B$ with a closed manifold $B$ as base and generators $S=\{[\g_1],\dots,[\g_n]\}$ of  $\pi_1(B,b)$ and fibre $F$ as above, the \emph{Schreier space} $\mathcal{S}$ is the subspace of $E$ formed by taking the union of all lifts of the paths $\{\g_1,\dots,\g_n\}$ to paths in $E$ starting at points of $F$.
\end{Def}

It is then straightforward that $L_x \cap \mathcal{S}$ in the leaf topology is the Schreier graph of the coset space $\pi_1(B,b)/\mathcal{K}_x$, and so has the same number of ends as the leaf $L_x$.

We shall refer to the sets of the form $L \cap \mathcal{S}$ for a
leaf $L$ of $\F$ as \emph{leaves} of $\mathcal{S}$ and a set formed
by a union of leaves will be called \emph{saturated} as with
foliations. A compact saturated subset of $\mathcal{S}$ is
\emph{minimal} if it is the closure of each of its leaves. If the
sequence $\{x_i\}_{i=1}^\infty$ converges to $x$ in $\mathcal{S}$,
then the corresponding plaques in a foliation chart of $x$ converge
to the plaque of $x$. Hence, since the leaves of $\mathcal{S}$ are
formed by lifting paths in $B$, a set is saturated in $\mathcal{S}$
if and only if the corresponding set is saturated in $E$ and a set
in $\mathcal{S}$ is minimal if and only if the corresponding set is
a compact minimal subset of $E$.  Since the leaves of $\mathcal{S}$
are connected, minimal sets are compact and connected; that is, they
are continua. Thus, we refer to a minimal subcontinuum $\Lambda$ of
$\mathcal{S}$ as the \emph{Schreier continuum} of the corresponding
minimal set of $E$. If the foliation of the bundle $p:E \to B$ is
minimal, the notions of the Schreier space and the Schreier
continuum are interchangeable.

\section{Schreier continua for solenoids}\label{sec:schreiercontsolenoids}


In this section we prove Theorem \ref{Theor-one}, that is, for any $n> 1$ we construct a compact foliated space $X_n$ such that the union of leaves with $1$ end is a residual subset of $X_n$, the union of leaves with $2$ ends forms a meager subset of $X_n$ and the number of such leaves is uncountable, and there is a single leaf with $2n$ ends.

We are going to construct these examples as inverse limits of sequences of finite coverings of a genus $2$ surface, i.e. \emph{weak solenoids}. For convenience we recall basic facts about solenoids in Section \ref{subsec-weaksolenoids}, where we also give some algebraic examples of computations of end structures. The proof of Theorem \ref{Theor-one} is obtained on the basis of the Schori solenoid \cite{S}, an example where a direct algebraic technique is not
sufficient.

The proof of Theorem \ref{Theor-one} for $n=2$ is detailed in Section \ref{sec:schoriexample}. In Section \ref{subsec:generalisedSchori} we show how
to modify Schori's solenoid in such a way that the exceptional leaf
has any given even number of ends. In Section \ref{subsec:RTsolenoid} we describe the Schreier continuum for the Rogers-Tollefson example (see also Example \ref{Example:RT}). In Section \ref{subsec-nonselfsimilar} we give an example of a solenoid, based on the Schori construction, where the group of automorphisms of the fibre acts non self-similarly.

\subsection{Weak solenoids as foliated bundles}\label{subsec-weaksolenoids}

A general class of examples of foliated bundles with a Cantor set fibre can be found in weak solenoids. Here, a solenoid $X_\infty $
will refer to the inverse limit
   \begin{align}  X_\infty = \lim_{\longleftarrow} \{ X_k, f^k_{k-1}, \mb{N}_0 \} \end{align}
of an inverse sequence of closed manifolds $X_k$ where each bonding map $f^k_{k-1}:X_k \rightarrow X_{k-1}$ is a covering map of index greater than one. The projection $f_0:X_\infty \rightarrow X_0$ is a fibre bundle projection
with a profinite structure group and Cantor set fibre (see Fokkink and Oversteegen \cite[Theorem 33]{FO}). When a solenoid is homogeneous (i.e., when the homeomorphism group acts transitively on the solenoid), the solenoid can be represented as the inverse limit of an inverse sequence where each covering map $f^k_0$ is regular~\cite{FO}, and in this case the associated bundle is principal (see McCord \cite[Theorem 5.6]{Mc}). Thus, a solenoid can be considered a natural generalisation of  covering spaces of closed manifolds where the fibre is no longer discrete but instead totally disconnected, and where regular covering spaces correspond to homogeneous solenoids. Locally, a solenoid is homeomorphic to
$D \times \text{\rm Cantor Set}$, where $D$ is a Euclidean disk of the same the dimension as the base manifold $X_0$.

Denote by $G_k$ the image of the homomorphism
  $$(f^k_0)_*:\pi_1(X_k,x_k) \rightarrow \pi_1(X_0,x_0),$$
then the kernel at $x_\infty = (x_k)$, as defined in the previous section, is given by
\[
\mathcal{K}_{x_\infty}=\bigcap_{k \in \mb{N}_0} \, G_k .
\]

Thus, the Schreier continuum associated to $X_\infty$ and some fixed finite set of generators $S$ of $\pi_1(X_0,x_0)$ is given by
\[
\Lambda=\lim_{\longleftarrow} \{ \Lambda_k,\, \nu_{k-1}^k,\, k \in \mb{N}_0 \}.
\]
where $ \Lambda_k$ is the Schreier diagram for the coset space $\pi_1(X_0,x_0)/G_k$. Alternatively, consider $\Lambda_0$ to be the image of the Schreier space $\mc{S}$ under the covering projection $f_0$, and each $\Lambda_k$ as the lift of $\Lambda_0$ to $X_k$ by the covering map $f_0^k$.

\begin{Example}\label{Example:dyadics}
{\rm In the case that the solenoid is a principal bundle and all
path-connected components are homeomorphic, the end structure can
typically be found directly by analysing the kernels $\mc{K}_x$ at
points $x \in F$, and all possible end structures do occur. For
example, let $X_k = \mb{S}^1$ and each $f^k_{k-1}$ be a $2$-fold
covering map. The inverse limit
  \begin{align*}  \Sigma_2 & = \lim_{\longleftarrow} \{ \mb{S}^1, f^k_{k-1}, \mb{N}_0 \}   \end{align*}
is usually called the dyadic solenoid. The universal cover of each
leaf is $\mb{R}$ and for each $x \in \Sigma_2$ the kernel $\mc{K}_x$
is trivial. Therefore, each leaf in $\Sigma_2$ is homeomorphic to
$\mb{R}$ and so has $2$ ends. The dyadic solenoid has a self-similarity structure and its ends could be treated in the context of self-similar group actions as in \cite{BDN}; however, if the bonding maps $f^k_{k-1}$ are distinct primes and the resulting solenoid has no self-similarity, the results of  \cite{BDN} do not apply but this solenoid can be treated by our technique in the same way as the dyadic solenoid.

Generalising this example to $2$
dimensions one obtains a solenoid $\Sigma_2 \times \Sigma_2$ which
fibres over a $2$-torus $\mb{S}^1 \times \mb{S}^1$. The universal
cover of each leaf is $\mb{R}^2$, and for each $x \in \Sigma_2 \times \Sigma_2$
the kernel $\mc{K}_x$ is trivial and so each leaf has $1$ end. It
is known that a free group on $2$ generators $F_2$ can be realised
as the fundamental group of a $3$-manifold, and $F_2$ is
geometrically residually finite (see Scott~\cite{Sc1} for general
results about residually finite groups). Thus, choosing an
appropriate decreasing chain of subgroups $G_i$ of finite index with $\cap_i G_i =
\emptyset$ with each $G_i$ normal in $F_2$, one can construct a
solenoid in which each leaf has a Cantor set of ends.}
\end{Example}

\begin{Example}\label{Example:RT}
{\rm More interesting examples of ends occur when the solenoid does
not have the structure of the principal bundle and the structure of
the set of ends vary from leaf to leaf. One such example is given by
Rogers and Tollefson \cite{RT}. Their solenoid is the inverse limit
of $2$-fold coverings of the Klein bottle $K_i$ by itself
  $$K_\infty = \lim_{\longleftarrow} \{ K_i, f^i_{i-1},\mb{N}_0 \}.$$
The solenoid $K_\infty$ is double covered by the solenoid $\mb{S}^1
\times \Sigma_2$  with non-trivial covering transformation of the
covering map
   \begin{align*} p_\infty &: \mb{S}^1 \times \Sigma_2 \to K_\infty \end{align*}
represented by the involution $\beta \times \alpha_\infty:  \mb{S}^1 \times \Sigma_2 \to  \mb{S}^1 \times \Sigma_2$, where $\beta$ is the rotation in the circle by $\pi$ and $\alpha_\infty$ is the map induced by the reflections $\alpha_i:\mb{S}^1 \to \mb{S}^1$. The involution $\beta \times \alpha_\infty$ has a unique fixed point $x$, so it fixes a single path-connected component of $\mb{S}^1 \times \Sigma_2$ and is a homeomorphism when restricted to any other path-connected component. Thus every path-component in $K_\infty$ other than the exceptional one is homeomorphic to $\mb{S}^1 \times \mb{R}$ and has $2$ ends. The exceptional path-component is homeomorphic to the quotient space
  \begin{align*} \mb{S}^1 \times \mb{R} / \big\{ (t,s) \sim (\beta(t),-s) \big\}, \end{align*}
it is non-orientable and has $1$ end.
}
\end{Example}

\subsection{Schori solenoid}\label{sec:schoriexample}

The Schori example is the solenoid
  \begin{align*}X_{\rm{Sch}} & = \lim_{\longleftarrow} \{ X_k, f^{k+1}_k, \mb{N}_0\}, \end{align*}
where $X_0$ is a genus $2$ surface, and for every $k \in \mb{N}_0$
the bonding map $f^{k+1}_k:X_{k+1} \to X_k$ is a $3$-to-$1$
non-regular covering projection. We recall briefly the construction from \cite{S}.

The construction is by induction. Let $X_0$ be a genus $2$ surface. For $n > 0$ let an $n$-handle be a $2$-torus with $n$ mutually disjoint open disks taken out, then $X_0$ is a union of two $1$-handles $H_0$ and $F_0$ which intersect along their boundaries. Let $C_0$ and $D_0$ be simple closed curves in $H_0$ and $F_0$ respectively (see Figure \ref{fig:M0-pic}, a)). For $k> 0$ let $X_k$ be a genus $m_k=3^k+1$ surface with two $2^k$-handles $H_k$ and $F_k$ distinguished, and let two simple closed curves $C_k$ and $D_k$ be chosen in $H_k$ and $F_k$ respectively. Then the next component in the sequence is obtained in the following manner: take out the curves $C_k$ and $D_k$ from $X_k$ and pull the cut handles apart by an appropriate homeomorphism to obtain a cut surface $\widehat{X}_k$. Let $\bar{X}_k = {\rm{Cl}}(\widehat{X}_k)$, and denote by $C_k'$ and $C_k''$ (resp. $D_k'$ and $D_k''$) the boundary circles in $H_k$ (resp. $F_k$) (see Figure \ref{fig:M0-pic}, b)). Consider three copies $\bar{X}_k^1$, $\bar{X}_k^2$ and $\bar{X}_k^3$ of $\bar{X}_k$, so that for $i=1,2,3$ $\bar{H}_k^i$ and $\bar{F}_k^i$ are the cut handles in $\bar{X}_k^i$, and ${C_k^i} '$ and ${C_k^i}''$ (resp. ${D_k^i} '$ and ${D_k^i}''$) are the boundary circles in $H_k^i$ (resp. $F_k^i$). Then in $H_k^i$-handles identify ${C_k^1}'$ with ${C_k^1}''$, ${C_k^2}'$ with ${C_k^3}''$, and ${C_k^3}'$ with ${C_k^2}''$, and in the $F_k^i$-handles identify ${D_k^1}'$ with ${D_k^2}''$, ${D_k^2}'$ with ${D_k^1}''$, and ${D_k^3}'$ with ${D_k^3}''$ (see Figure \ref{fig:M0-pic}, c)). Denote the identification space by $X_{k+1}$, and let $H_{k+1}$ (resp. $F_{k+1}$) be the image of $H_k^2 \bigsqcup H_k^3$ (resp. $F_k^1 \bigsqcup F_k^2$) under the identification map (see Figure \ref{fig:M0-pic}, d)). Define the mapping $f^{k+1}_k:X_{k+1} \to X_k$ by sending a point $(x,i) \in \bar{X}_k^i$ to $x \in X_k$. The obtained map is a $3$-to-$1$ covering projection, and set
 $$X_{\rm{Sch}} = \lim_{\longleftarrow} \{X_k,f^{k+1}_k,\mb{N}_0\}.$$

\begin{figure}
\centering
\includegraphics [width=105mm] {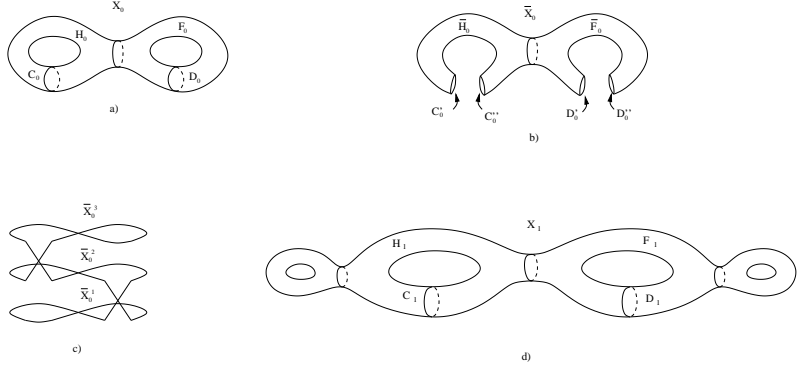}
\caption{Construction of the Schori example: a) Choice of the handles $H_0$ and $F_0$ and closed curves $C_0$ and $D_0$ in $X_0$, b) The cut surface $\bar{X}_0$, c) Identifications between $\bar{X}^i_0$, $i=1,2,3$. Each $\bar{X}^i_0$ is represented by a cut copy of figure $8$, and identifications are depicted with straight lines, d) The surface $X_1$ and the choice of the handles $H_1$ and $F_1$ and closed curves $C_1$ and $D_1$.}
 \label{fig:M0-pic}
\end{figure}

\subsection{The Schreier continuum in the Schori example}\label{subsec:SchSchori}

Consider the Schori solenoid $X_{\rm{Sch}}$, and let $x_0 \in H_0 \cap F_0$. For each $k >0$ there is a unique point $x_k \in X_k$ such that $x_k  \in H_k \cap F_k$ and
  $$f^k_0 (x_k) = f^k_{k-1} \circ f^{k-1}_{k-2} \circ \cdots \circ f^1_0(x_k) = x_0,$$
so $(x_k)_{k=0}^\infty \in X_{\rm{Sch}}$. In the sequel we omit the subscript and superscript in the notation for a point in $X_{\rm{Sch}}$ writing $(x_k)$ instead of $(x_k)_{k=0}^\infty$. Denote by $f_k: X_{\rm{Sch}} \to X_k$ the projection, and by $F$ the fibre $f_0^{-1}(x_0)$ of the bundle $p_0:X_{\rm{Sch}} \to X_0$ at $x_0$.

\subsubsection{The group chain and the Schreier continuum}

Let $G_0 =\pi_1(X_0,x_0)$, and for $k> 0$ denote $G_k = {f^k_0}_* \pi_1(X_k,x_k)$. As it was shown in Section \ref{sec:schreiercontsolenoids}, the Schreier continuum for a pointed space $(X_{\rm{Sch}},(x_k))$ is constructed as an inverse sequence of Schreier diagrams $\Lambda_k'$ of pairs of groups $(G_0,G_k)$. So the first step in the construction is to compute the group chain $G_0 \supset G_1 \supset G_2 \supset \ldots$

Choose loops $a,b,\alpha,\beta$ representing the generators of the fundamental group $\pi_1(X_0,x_0)$ in the standard way (see Section \ref{subsec:schreierspace}) with a relation
  \begin{align*}{\rm{rel}}_0 & =[a, \alpha][b,\beta] = a \alpha a^{-1} \alpha^{-1} b \beta b^{-1} \beta^{-1},\end{align*}
where concatenation denotes the usual multiplication of paths. Calculation of $G_k$ by an inductive procedure gives the following. For $k=0$, distinguish the following subsets of generators of $G_0$,
\begin{subequations}\label{eq:recformula}
 \begin{align}\label{eq:recformula0} S_{0ab} & = \emptyset, &  S_{0ba} & = \emptyset \\ S_{0a} & = \{a, \alpha\},&  S_{0b} & = \{b, \beta\}. \end{align}

\begin{figure}
\centering
\includegraphics [width=105mm] {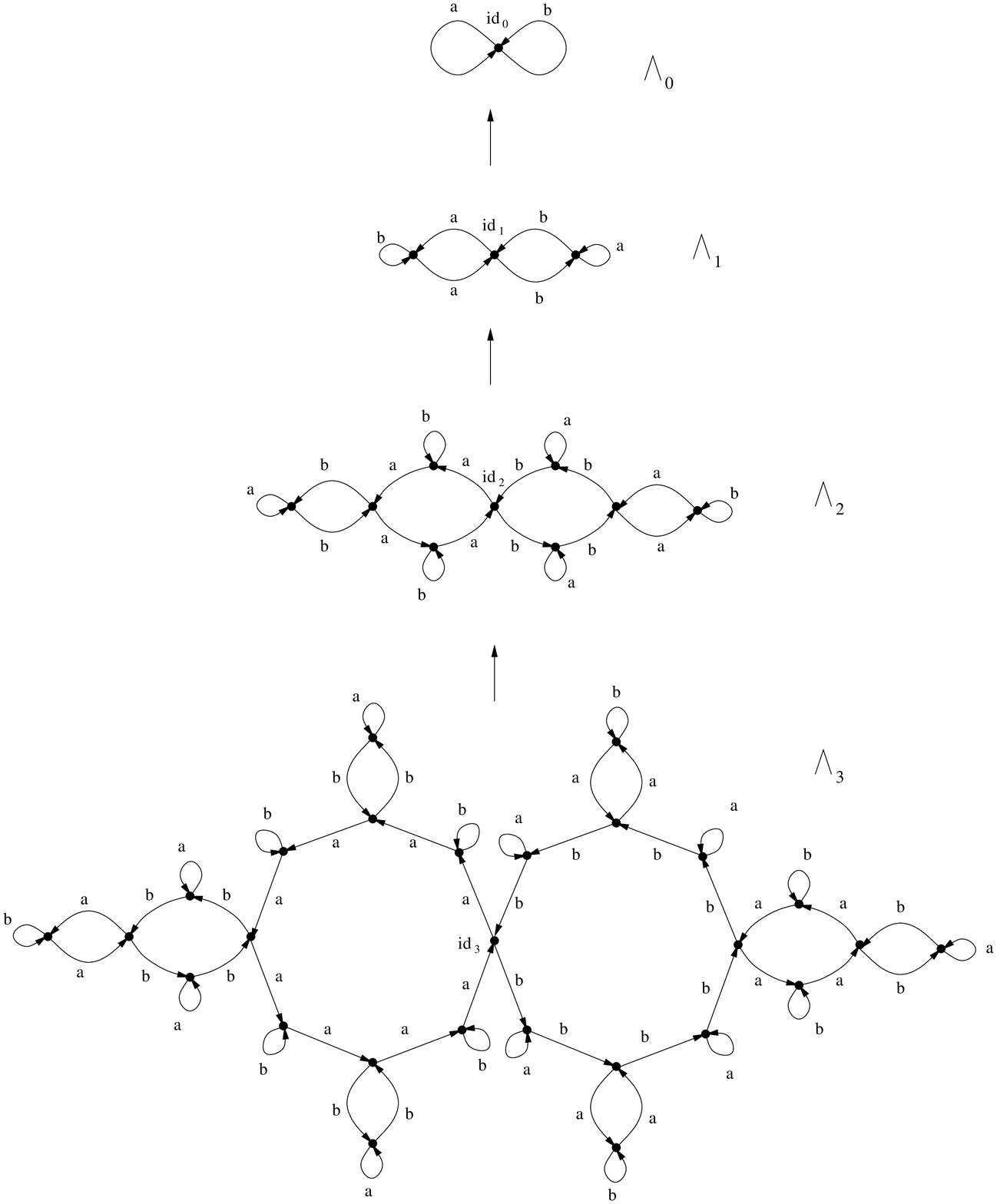}
\caption{Graphs $\Lambda_i$ for $i=0,1,2,3$.}
 \label{fig:SchreierSeq-beginning}
\end{figure}

For $k \geq 1$ define
  \begin{align} \nonumber S_{kab} & = S_{(k-1)ab} \cup a^{2^{k-1}} S_{(k-1)ab} a^{-2^{k-1}} \cup a^{2^{k-1}} S_{(k-1)b} a^{-2^{k-1}}, \\
                \nonumber  S_{kba} & = S_{(k-1)ba} \cup b^{2^{k-1}} S_{(k-1)ba} b^{-2^{k-1}} \cup b^{2^{k-1}} S_{(k-1)a} b^{-2^{k-1}}, \\
                 \label{eq:recursiveformula} S_{ka} & = \{ a^{2^k}, \alpha\} \cup S_{k a b},  \\
                \nonumber  S_{kb} & = \{ b^{2^k}, \beta\} \cup S_{k b a}.
 \end{align}
Then for $k\geq 0$ we have
  \begin{align*} G_k & = \langle a^{2^k}, \alpha, b^{2^k}, \beta, S_{kab}, S_{kba} \ | \ {\rm{rel}}_k=\id, {\rm{rel}}_0=\id \rangle, \end{align*}
where ${\rm{rel}}_k$ is the corresponding relation for the $m_k$-genus surface. The subgroup $G_k$ has index $3^k$ in $G_0$.
\end{subequations}

Let $\bar{S}_{0a}$ and $\bar{S}_{0b}$ be sets of formal inverses of elements in $\bar{S}_{0a}$ and $\bar{S}_{0b}$ respectively, and denote $S = S_{0a} \cup S_{0b} \cup \bar{S}_{0a} \cup \bar{S}_{0b}$. For each $k \geq 0$ let $\Lambda_k'$ be the Schreier diagram of the coset space $G_0/G_k$ constructed with respect to $S$ as in Section \ref{sec:schreiercontsolenoids}, and $\Lambda' = \lim\{\Lambda_k',f^{k+1}_k,\mb{N}_0\}$ be the corresponding Schreier continuum. Let $v = (v_k)$ be a vertex of $\Lambda'$, and denote by $L_v$ its path-connected component.

\subsubsection{The simplified Schreier continuum}

We call an edge labelled by $\lambda$ a $\lambda$-edge and write $e_\lambda$. Then $s(e_\lambda)$ is the starting vertex of the edge, and $t(e_\lambda)$ is the ending vertex. If $s(e_\Lambda) = t(e_\lambda)=v$, $e_\lambda$  is called a $\lambda$-loop based at $v$. If it is important to keep track the starting
point of an edge, we write $e_\lambda^v$. A path in $\Lambda'$ is a continuous map $\gamma:[0,1] \to \Lambda'$, so its image is contained in a path-connected component of $\Lambda'$. We allow paths to start and end in the interior of edges. A subgraph $L'$ of $\Lambda'$ is called path-connected, if for any $v,w \in L'$ there is an edge-path $\gamma \subset L'$ which starts at $v$ and ends at $w$.

We notice that any vertex $w=(w_k) \in \Lambda'$ is a base point of $\alpha$-, $\alpha^{-1}$-, $\beta$- and $\beta^{-1}$-loops, since by construction any $\alpha$-, $\alpha^{-1}$-, $\beta$- and $\beta^{-1}$-loop based at $w_k \in \Lambda_k'$, $k \in \mb{N}_0$ lifts to a loop in $\Lambda_{k+1}'$ at $w_{k+1}$ labelled by the same letter. Given an exhaustion of a path-connected component $L'$ of $w$ by compact sets, eventually every such loop would be contained in a compact sets and will not make any contribution to the number of ends of $L$. This motivates the following lemma.

\begin{Lemma}\label{Lemma:simplifiedSchreier}
Let $S' = \{\alpha,\alpha^{-1},\beta,\beta^{-1}\}$, for $k \in \mb{N}_0$ let
  \begin{align*}\Lambda_k & = \Lambda_k' \backslash \bigcup_{w_k \in V(\Lambda_k')} \Big( \bigcup_{\lambda \in S'}{\rm{int}} \ e_\lambda^{w_k} \Big), \end{align*}
and $r_k: \Lambda_k' \to \Lambda_k$ be a retraction which maps each $\lambda$-loop $e^{w_n}_\lambda$, where $\lambda \in S'$, to its base point $w_n$ and is the identity on the rest of $\Lambda_k'$.  Denote by $i_k:\Lambda_k \to \Lambda_k'$ the inclusion, and define a continuous map $\nu^k_{k-1}:\Lambda_k \to \Lambda_{k-1}$ by
   \begin{align}\label{eq:numappings} \nu^k_{k-1}(x) & = r_{k-1} \circ {\nu'}^k_{k-1} \circ i_k(x). \end{align}
Then $\{\Lambda_k,\nu^k_{k-1},\mb{N}_0\}$ is an inverse sequence of topological spaces, and there is an induced map
  \begin{align}\label{eq:retrmap} r_\infty&: \Lambda' \to \Lambda = \lim_{\longleftarrow} \{ \Lambda_k,\nu^{k+1}_k, \mb{N}_0\}, \end{align}
which is a bijection between the sets of path-connected components of $\Lambda'$ and $\Lambda$. Moreover, for any vertex $v \in \Lambda'$ the path-connected components $L_v'$ and $r(L_v')$ have the same number of ends.
\end{Lemma}
\begin{proof} The mappings \eqref{eq:numappings} satisfy $\nu^k_i = \nu^j_i \circ \nu^k_j$ for any $i < j < k$, and $\Lambda_k$ are compact, so the inverse limit $\Lambda$ exists and is non-empty \cite{RZ}. The retractions $r_k$ satisfy $r_k \circ {\nu'}^{k+1}_k = \nu^{k+1}_k \circ r_{k+1}$ and are surjections onto compact spaces $\Lambda_k$, so they induce a continuous surjection \eqref{eq:retrmap} on the inverse limit spaces. On the other hand, the inclusions $i_k$ induce a continuous mapping $i_\infty:\Lambda \to \Lambda'$ such that $r_\infty \circ i_\infty = \id_\Lambda$. It follows that \eqref{eq:retrmap} is a bijection between path-connected components of $\Lambda'$ and $\Lambda$, and $r_\infty(L_v') = L_v$.

Let $d'$ be a complete length metric on the leaves of $\Lambda'$ such that the restrictions of $\nu_k'$ to leaves is a local isometry. In particular, each edge in $\Lambda'$ has length $1$ in the length structure associated to the metric. Denote by $d$ the induced length metric on the subspace $\Lambda$ of $\Lambda'$. Denote by $B'(w,s)$ and $B(w,s)$ compact balls of radius $s$ about a vertex $w$ with respect to $d'$ and $d$ respectively, and consider the exhaustion $\{B'(v, z + \frac{1}{2}) \}_{z \in \mb{N}}$ of $L_v'$ by closed balls.

Notice that if $\lambda \in S'$ and $e_\lambda^w$ is a loop, then for any $x \in {\rm{int}} \ e^w_\lambda$ we have $d ' (w,x) \leq \frac{1}{2}$. Therefore, for any $z \in \mb{N}$ the intersection $B'(v, z + \frac{1}{2}) \cap e^w_\lambda$ is either empty or contains the edge $e^w_\lambda$. Therefore, there is a bijection between the sets of path-connected components of the complements $L_v' \backslash B'(v, z + \frac{1}{2})$ and $L_v \backslash r_\infty(B'(v, z + \frac{1}{2}))$. The sequence
  \begin{align*} \Big\{r_\infty(B'(v, z + \frac{1}{2}) \Big\}_{z \in {\tiny \mb{N}}} & = \Big\{ B(v, z + \frac{1}{2}) \Big\}_{z \in {\tiny \mb{N}}} \end{align*}
is an exhaustion of $L_v$, and it follows that $L_v'$ and $L_v$ have the same number of ends. 
\end{proof}

In the case of the Schori example $\Lambda_0$ is a figure $8$, and each $\Lambda_k$ is the Schreier graph of a coset space $F_2/G_k'$, where $F_2$ is a free group on two generators and $G_k' \subset F_2$ is a subgroup of finite index (see Figure \ref{fig:SchreierSeq-beginning}) with the following set of generators: denote by $\bar{S}_{0a} = \{a\}$, $\bar{S}_{0b} = \{b\}$, for $k>1$ define $S_{kab}$ and $S_{kba}$ as in \eqref{eq:recursiveformula}, and
  \begin{align*} S_{ka} & = \{ a^{2^k}\} \cup S_{kab}, & S_{kb} = \{ b^{2^k} \} \cup S_{kba}. \end{align*}
Then
  \begin{align*} G_k' & = \langle a^{2^k}, b^{2^k}, S_{kab}, S_{kba} \rangle. \end{align*}

\subsubsection{Ends of special path-connected components in the Schori example}\label{subsubsec:leavesSchori}

As in Lemma \ref{Lemma:simplifiedSchreier}, we consider the sequence of topological spaces $\Lambda_k$, $k \in \mb{N}_0$, with length structure induced from $\Lambda_k'$ and the corresponding length metric $d_k$. Since $\Lambda_k$ is compact, $d_k$ is bounded. In fact, we can obtain a more precise estimate, which will be of use later.

For $k \in \mb{N}_0$ we denote by $\id_k$ the coset $[\id] \in F_2/G_k'$, and by $a^m$ (resp. $b^m$), $m = 1,\ldots,2^k-1$, the coset $[a^m]$ (resp. $[b^m]$). Denote by $\mc{A}_k$ the connected component of $\Lambda_k - \{\id_k\}$, containing vertices $a^1, \ldots, a^{2^k-1}$, and by $\mc{B}_k$ the other connected component. Then $\mc{A}_k \cup \{\id_k\}$ and $\mc{B}_k \cup \{\id_k\}$ are compact connected subspaces of $\Lambda_k$.

\begin{Lemma}\label{Lemma:metriconGamma}
The restriction of the length metric $d_k$ to $\mc{A}_k \cup \{\id_k\}$ or to $\mc{B}_k \cup \{\id_k\}$ has the least upper bound $N_k= 2^k - \frac{1}{2}$, and for any two positive integers $m,\ell  < 2^k$
  \begin{align*} d_k(a^\ell,a^m)& = \min \{|m-\ell|, 2^k - |m-\ell| \}, \end{align*}
and the same holds for $d_k(b^\ell,b^m)$.
\end{Lemma}

\begin{Cor}
For any positive integers $m,\ell < 2^{k-1}$
   \begin{align*} d_k(a^\ell,a^m)& = |m-\ell|. \end{align*}
\end{Cor}

Denote by $\id$ the point $(\id_k) \in \Lambda$.

\begin{Cor}\label{Cor:corpaths}
If $m,\ell$ are positive integers, then for $(a^m),(a^\ell) \in L_\id$
    \begin{align*} d\big((a^m),(a^\ell)\big) & = |m - \ell|, \end{align*}
and there is a corresponding geodesic in $L_\id$ consisting only of $a$- or only of $a^{-1}$-edges.
\end{Cor}

In the following proposition we will also use a notation $a^{-m}$ for a coset $[a^{-m}] \in F_2/G_k'$, where $m \in \mb{N}_0$, so that vertices in $F_2/G_k'$ now have two names since $[a^\ell] = [a^{-m}]$ with possibly  $\ell \ne m$. We will see that $(a^{-m}) \ne (a^\ell)$ in $L_\id$ for any $m,\ell >0$ in the proof of Proposition \ref{Prop:pathcomponentofid}.

\begin{Prop}\label{Prop:pathcomponentofid}
The pathwise connected component $L_\id$ of $\Lambda$ has four ends.
\end{Prop}
\begin{proof} The idea is to choose a cofinal sequence of compact subsets in $L_\id$ so that their complements are easy to handle. Such a sequence is a sequence of compact balls $B(\id,N_k)$ where $N_k = 2^k - \frac{1}{2}$ (see Lemma \ref{Lemma:metriconGamma}).

Let $A \subset V(L_\id)$ be a subset of the set of vertices of $L_\id$. We say that a subgraph $\Gamma_A$ is associated to $A$, if $\Gamma_A$ contains $A$ and if $\Gamma_A$ contains all edges $e_\lambda \in L_\id$ with both starting and ending vertices in $A$. We claim that for each $k \in \mb{N}_0$ and for each of the following subsets of $V(L_\id)$
  \begin{align} \label{eq:connectedsets} A_{k}^+ & = \{(a^m) \ | \  m > N_k\}, & A_{k}^- & = \{(a^{-m}) \ | \ m > N_k \}, \\
          \nonumber   B_{k}^+ & =  \{(b^m) \ | \ m > N_k \}, & B_{k}^-  & = \{(b^{-m}) \ | \ m > N_k \} \end{align}
the associated graph lies in a distinct unbounded connected component of the complement $L_\id \backslash B(\id,N_k)$, and each $L_\id \backslash B(\id,N_k)$ has at most four unbounded connected components.

Clearly $\Gamma_{A_k^+}$ is path-connected. Since $B(\id,N_k)$ contains an interior point of an edge $e_\lambda$ if and only if it contains at least one of its vertices,
  $$\Gamma_{A_k^+} \cap B(\id,N_k)  = \emptyset$$
and so $\Gamma_{A_k^+}$ is contained in a path-connected component of $L_\id \backslash B(\id,N_k)$. By Corollary \ref{Cor:corpaths} geodesics between vertices in $A_k^+$ are contained in $\Gamma_{A_k^+}$ and they can be of arbitrarily large length, so $\Gamma_{A_k^+}$ is contained in an unbounded path-connected component of the complement. By a similar argument each of the sets $A^-_k$, $B^+_k$, $B^-_k$ is contained in an unbounded connected component of the complement $L_\id \backslash B(\id,N_k)$.

We show that $\Gamma_{A_{k}^+}$ and $\Gamma_{A_{k}^-}$ lie in
different path-connected components of $L_\id \backslash
B(\id,N_k)$. Assume to the contrary that there is a path $\gamma$
between $(a^m)$ and $(a^{-n})$ entirely contained in $L_\id
\backslash B(\id,N_k)$. Choose $j \in \mb{N}_0$ so that for each $i
\geq j$ the path $\gamma$ is contained in the fundamental domain of
the covering map $\nu_i:L_\id \to \Lambda_i$, then its projection
$\gamma_i = \nu_i \circ \gamma$ is contained in $\Lambda_i
\backslash B_i(\id_i,N_k)$, where $B_i(\id_i,N_k)$ is a closed ball
about $\id_i$ of radius $N_k$ with respect to the length metric
$d_i$ on $\Lambda_i$. Then
  \begin{align*} \ell(\gamma) & \geq \ell_i(\gamma_i) \geq  2^i-(m + n),\end{align*}
where $\ell(\gamma)$ and $\ell_i(\gamma_i)$ are the lengths of paths in the length structures associated to $d$ and $d_i$ respectively. Since $i$ can be arbitrary large, $\ell(\gamma)$ cannot be finite. It follows that $\Gamma_{A_{k}^+}$ and $\Gamma_{A_{k}^-}$ (and, by a similar argument, $\Gamma_{B_{k}^+}$ and $\Gamma_{B_{k}^-}$) are in different path connected components of $L_\id \backslash B(\id,N_k)$.

Next, suppose $\Gamma_{A^+_k}$ and $\Gamma_{B^+_k}$ (resp. $\Gamma_{B^-_k}$) are in the same path-connected component of $L_\id \backslash B(\id,N_k)$ and choose $j \in \mb{N}$ so that a path $\gamma$ between $(a^m) \in A^+_k$ and $b^n \in B^+_k$ (resp. $b^{-n} \in B^-_k$) is contained in the fundamental domain of the covering map $\nu_j:L_\id \to \Lambda_j$. Then $\gamma_j = \nu_j \circ \gamma$ is contained in a path-connected component of $\Lambda_j \backslash B_j(\id_j,N_k)$, which implies that either $a^m,b^n \in \mc{A}_j$ or $a^m,b^n \in \mc{B}_j$ (resp. $a^m,b^{-n} \in \mc{A}_j$ or $a^m,b^{-n} \in \mc{B}_j$), a contradiction. Therefore, $\Gamma_{A^+_k}$ and $\Gamma_{B^+_k}$ (resp. $\Gamma_{B^-_k}$) are in different path-connected components of $L_\id \backslash B(\id,N_k)$. Repeat the same argument for $\Gamma_{A^-_k}$. Conclude that $L_\id \backslash B(\id,N_k)$ has at least four unbounded path-connected components.

We show that $L_\id \backslash B(\id,N_k)$ has at most $4$
path-connected components. Note that each $B(\id,N_k)$ is precisely
the fundamental domain of the covering map $\nu_k: L_\id \to
\Lambda_k$. Then for $i > k$ the coboundary of the set of vertices
$V(B_i(\id_i,N_k))$ consists of exactly $4$ edges. Let $v = (v_n)
\in L_\id \backslash B(\id,N_k)$, and $\gamma$ be a path in $L_\id$
between $\id$ and $v$. Choose an integer $j>k$ such that $i>j$
implies that the path $\gamma$ is contained in the fundamental
domain of the covering map $\nu_i:L_\id \to \Lambda_i$, and so
projects faithfully on $\Lambda_i$. Since $ \id_i = \nu_i(\id) \in
B_i(\id_i,N_k)$ and $v_i \in \Lambda_i \backslash  B_i(\id_i,N_k)$,
the projection $\gamma_i = \nu_i \circ \gamma$ contains an edge of
the coboundary of $V(B_i(\id_i,N_k))$. It follows $v$ is in a
path-connected component containing one of the sets in
\eqref{eq:connectedsets}. 
\end{proof}

Let $q_0 = \id_0$, $q_1 = a$  and for $k \in \mb{N}_0$, $k>1$ define
  \begin{align}\label{eq:pathconnectedoneend} q_{2k} & = b^{2^k} q_{2k-1}, & q_{2k+1} & = a^{2^k+1} q_{2k}. \end{align}
Then $\nu^k_{k-1} (q_k) = q_{k-1}$, and $q=(q_k)$ is a vertex in $\Lambda$.

\begin{Lemma}\label{Lemma:oneendcomponent}
A path-connected component $L_q$ of $\Lambda$ has $1$ end.
\end{Lemma}
\begin{proof} We claim that for each $k > 0$ the set $U_k=L_q \backslash B(q,N_k - \frac{1}{2})$ is path-connected. First notice that for any $i > k$ we have
  \begin{align*}\nu_i\Big(B\big(q,N_k - \breukje{1}{2}\big)\Big) & = B_i \big(q_i,N_k - \breukje{1}{2} \big),\end{align*}
and if $d(q,v)> N_k - \frac{1}{2}$, then $\nu_i(v) \in B_i\big(q_i, N_k - \frac{1}{2}\big)$ for at most a finite number of $i \in \mb{N}_0$. Denote by $\mc{Q}_i$ and $\mc{P}_i$ path-connected components of $\Lambda_i \backslash \{\id_i\}$ so that $q_i \in \mc{Q}_i$, and notice that if $i > k$ then $\nu_i\Big(B\big(q,N_k - \breukje{1}{2}\big)\Big) \subset \mc{Q}_i \cup \{\id_i\}$. By construction of $\Lambda_i$ and the choice of the radius of the balls (see Lemma \ref{Lemma:metriconGamma}) the complement $\big(\mc{Q}_i \cup \{\id_i\}\big) \backslash \nu_i\Big(B\big(q,N_k - \breukje{1}{2}\big)\Big)$ is path-connected.

Let $s=(s_n)$ and $t=(t_n)$ be points in $U_k$, and $\gamma$ and $\delta$ be paths joining $q$ with $s$ and $t$ respectively.  Choose $j > k$ so that
   \begin{align*}\max\{ \ell(\gamma),\ell(\delta)\}& < N_j - \frac{1}{2},\end{align*}
and for any $i>j$ we have $s_i,t_i \in \Lambda_i \backslash B_i\big(q_i,N_k - \frac{1}{2}\big)$. Then for all $i \geq j$ the paths $\gamma_i = \nu_i \circ \gamma$ and $\delta_i = \nu_i \circ \delta$ are contained in $\mc{Q}_i \cup \{\id_i\}$. Choose a path $c_j$ contained in $\mc{Q}_j \cup \{\id_j\}$ joining $s_j$ and $t_j$. Then for any $i>j$ the lift $c_i$ of $c_j$ with an endpoint $s_i$ has $t_i$ as another endpoint. So $U_k$ is path-connected and $L_q$ has one end. 
\end{proof}
Let $q_0' = \id_0$, $q_1' = b$  and define for $k \in \mb{N}_0$, $k>1$
  \begin{align}\label{eq:pathconnectedoneendagain} q_{2k}' & = a^{2^k} q_{2k-1}', & q_{2k+1}' & = b^{2^k+1} q_{2k}', \end{align}
then $\nu^k_{k-1} (q_k') = q_{k-1}'$, and $q'=(q_k')$ is a vertex in $\Lambda$. By the argument similar to that of Lemma \ref{Lemma:oneendcomponent} the path-connected component $L_{q'}$ has one end.

\subsubsection{Other path-connected components in the Schori example}\label{subsubsec:leavesSchoriother}

The remaining path-connected components of $\Lambda$ can be divided
into two groups: those containing a so-called `dyadic' point, and
those containing a so-called `flip-flopping' point. First we recall
and introduce some notation.

As in Lemma \ref{subsubsec:leavesSchori}, denote by $\mc{A}_k$ the path-connected component of $\Lambda_k \backslash \{\id_k\}$ containing cosets $a^i=[a^i] \in G_0/G_k$, and by $\mc{B}_k$ the other one. The vertex $\id_k$ has three preimages under $\nu^{k+1}_k: \Lambda_{k+1} \to \Lambda_k$, those being the vertices $\id_{k+1}, a^{2^k}, b^{2^k}$. Then $\mc{A}_{k+1} \backslash \{a^{2^k}\}$ consists of three path-connected components: the component $\mc{A}^+_{k+1}$ containing vertices $a^i$ for $i<2^k$, the component $\mc{A}^-_{k+1}$ containing vertices $a^i$ for $ i > 2^k$, and the remaining component $\mc{TA}_{k+1}$. Similarly, $\mc{B}_{k+1} \backslash  \{b^{2^k}\}$ consists of path-connected components $\mc{B}^+_{k+1}$, $\mc{B}^-_{k+1}$ and $\mc{TB}_{k+1}$.

Let $v = (v_k) \in \Lambda$ be a vertex, and assume that $v \notin L_\id$. If for some $k$ we have $v_k \in \mc{A}_k$ (resp. $v_k \in \mc{B}_k$), then either $v_{k+1} \in \mc{A}_{k+1}^+ \cup \mc{A}^-_{k+1}$ (resp. $v_{k+1} \in \mc{B}_{k+1}^+ \cup \mc{B}^-_{k+1}$) or $v_{k+1} \in \mc{TB}_{k+1}$ (resp. $v_{k+1} \in \mc{TA}_{k+1}$), and then the following situations are possible:
  \begin{enumerate}
  \item $v_k \in \mc{A}_{k}^+ \cup \mc{A}^-_{k} \cup \mc{B}_{k}^+ \cup \mc{B}^-_{k}$ for at most a finite number of $k$'s, then either $v_k \in L_q$ or $v_k \in L_{q'}$.
  \item $v_k \in \mc{TA}_{k} \cup \mc{TB}_{k}$ for at most a finite number of indices, then $v$ is called a \textit{dyadic} point. For example, let $j \in \mb{N}_0$ be odd, choose $0 \leq \ell_j < 2^j$ and define for $2k>j$
   \begin{align*}\ell_{2k} & = \ell_{2k-1} + 2^{2k-1}, & \ell_{2k+1} = \ell_{2k}. \end{align*}
Then $(a^{\ell_n})$ is a dyadic point.
  \item there is a cofinal subset $I \in \mb{N}_0$ such that for any $k \in I'$ we have $v_k \in \mc{TA}_k \cup \mc{TB}_k$, and the subset $\mb{N}_0 \backslash I$ is also cofinal. In this case $v$ is called a \textit{flip-flopping} point.
  \end{enumerate}

\begin{Lemma}\label{Lemma:dyadicleaf}
Let $v \in \Lambda$ be a dyadic point. Then $L_v$ has $2$ ends.
\end{Lemma}
\begin{proof} Suppose that for $k>n$ we have $v_k \in \mc{A}_k$ (resp. $v_k \in \mc{B}_k$). Then the distance
  \begin{align*} d_k(v_k,\{a^1, \ldots, a^{2^k-1}\}) & = \inf\{d_k(v_k,a^\ell) \ | \ 1 \leq \ell \leq 2^k - 1 \} \end{align*}
is constant for all $k > n$. Therefore, every dyadic point is in the
path-connected component of a point $\tilde{a}=(a^{\ell_k})$, where
$0 < \ell_k < 2^k$.

We construct an increasing sequence $K_0 \subset K_1 \subset K_2 \subset \ldots$ of compact subsets in $L_v$ such that $\tilde{a} \in K_0$ and for $i>0$ $L_v \backslash K_i$ has precisely $2$ unbounded path-connected components. Set $K_0  = B(\tilde{a},N_0)$ and suppose $K_{i-1}$ is given. Then there exists $j_i>0$ such that $K_{i-1} \subset B(\tilde{a},N_{j_i})$. Construct $K_i$ as follows. Lemma \ref{Lemma:metriconGamma} together with a simple computation show that $d_k(a^{\ell}_k,\id_k)$ increases with $k$, so there exists $n>0$ such that for all $k \geq n$ the projection $\nu_k\big(B(\tilde{a},N_{j_i})\big) \subset \mc{A}_k$. We can also assume that $2^i < 2^{n-2}$.  Denote $\pm m = \ell_n \pm (2^{j_i} - 1)$. The coboundary of $V\big(B_n(a^{\ell_n},N_{j_i})\big)$ contains the set of edges $W_n=\{e_a^{m}, e_{a^{-1}}^{m}, e_a^{-m}, e_{a^{-1}}^{-m}\} \subset \Lambda_n$, and a possibly empty set $E_n=\{e_{\lambda_1}, \ldots, e_{\lambda_{s_n}}\}$. There is a finite set of vertices $V_{E_n}=\{w_1^n,\ldots,w_m^n\} \in \Lambda_n$ such that a geodesic $\delta$ joining $a^{\ell_n}$ to $w_t$ contains an edge from $E_n$. The graph $K' = B_n(a^{\ell_n},N_{j_i}) \cup \Gamma_{V_{E_n}} \cup E_n$ is pathwise connected. For any $k>n$ there is a unique preimage $\tilde{K}' \subset \Lambda_k$ of $K'$ containing $a^{\ell_k}$, and its coboundary is precisely the set $W_k = \{e_a^{m}, e_{a^{-1}}^{m}, e_a^{-m}, e_{a^{-1}}^{-m}\} \subset \Lambda_k$. The inverse limit of $\tilde{K}'$ is a subset $K_i \subset L_v$ such that $K_i \owns \tilde{a}$.

Denote by $\bar{a}^m_k$ a vertex in $\Lambda_k$ corresponding to the coset $[a^ma^{\ell_k}]$, and define
 \begin{align} \label{eq:connectedsets2ended} A_i^+ & = \{\bar{a}^m=(\bar{a}^m_k) \ | \ m > N_{j_i}\}, & A_i^- & = \{\bar{a}^{-m}=(\bar{a}^{-m}_k) \ | \ m > N_{j_i} \}. \end{align}
The graphs $\Gamma_{A_k^+}$ and $\Gamma_{A_k^-}$ are clearly path-connected. To see that a path-connected component of $L_v \backslash K_i$ containing such graph is unbounded we show that $\bar{a}^m = \bar{a}^s$ if and only if $m=s$, which implies that $\Gamma_{A_k^\pm}$ contains geodesics of arbitrary length. Indeed, choose $n \in \mb{N}_0$ such that $\max\{m,s\} < 2^n$, then for all $k>n$ we have $\bar{a}^s_k = \bar{a}^t_k$ if and only if $s=t$. Next, consider the points $\bar{a}^{2^{j_i}} \in \mc{A}_i^+$, $a^{-2^{j_i}} \in \mc{A}_i^-$, and assume there exists a path $\gamma$ contained in $L_v \backslash K_i$ joining $\bar{a}^{2^{j_i}}$ and $\bar{a}^{-2^{j_i}}$. Choose $n > 0$ so big that for $k>n$  the projection $\nu_k(K_i) \subset \mc{A}_k$ and $2^{j_i} < 2^{n-2}$. The projection $\gamma_k = \nu_k \circ \gamma$ is a path joining $\bar{a}^{2^{j_i}}_k$ and $\bar{a}^{-2^{j_i}}_k$ and so must contain a sequence of $a$-edges. Assuming $a^{2^{j_i}}_k$ to be a starting point of $\gamma_k$, the first edge traversed by $\gamma_k$ must be $e_b$, $e_{b^{-1}}$ or $e_a$. In the two former cases $\gamma$ contains a loop based at $\bar{a}^{2^{j_i}}$, so we can as well assume that the first edge traversed by $\gamma_k$ is $e_a$. By a similar argument $\gamma_k$ must contain all $a$- or all $a^{-1}$-edges in $\Lambda_k$ such that $e_\lambda \cap \nu_k(K_i) = \emptyset$. Then we have the following estimate on the length $\ell(\gamma)$ and $\ell_k(\gamma_k)$
  \begin{align*} \ell(\gamma) & \geq \ell_k(\gamma_k) \geq  2^k-(2^{j_i}+1),\end{align*}
which implies that $\gamma$ has unbounded length. Therefore, $\Gamma_{A_k^+}$ and $\Gamma_{A^-_k}$ lie in different path-connected components of $L_v \backslash K_i$. We show that $L_v \backslash K_i$ has precisely two path-connected components. Assume $v \in L_v \backslash K_i$ and let $\delta$ be a path joining $v$ with $\tilde{a}$. Then $\delta_k$ contains an edge of $W_k$, and it follows that $v$ is in a path-connected component of $A_i^+$ or $A_i^-$. So there is at most $2$ path-connected components.
\end{proof}
\begin{Lemma}\label{Lemma:genericleaf}
Let $v \in \Lambda$ be a flip-flopping point. Then $L_v$ has $1$
end.
\end{Lemma}
\begin{proof} We construct a sequence of compact sets $K_0 \subset K_1 \subset K_2 \ldots$ such that $K_0 \owns v$ and $L_v \backslash K_i$ has one unbounded path-connected component. Set $K_0 = B(v,N_0)$, and suppose $K_{i-1}$ is given. Then there exists $j_i>0$ such that $K_{i-1} \subset B(v,N_{j_i})$. Since $d_k(a^{2^k},\id_k)$ and $d_k(b^{2^k},\id_k)$ increase with $k$, there is an $n>0$ such that for any $k\geq n$ such that $v_k \in \mc{TA}_k \cup \mc{TB}_k$ we have $B_k(v_k,N_{j_i}) \subset \mc{TA}_k \cup \mc{TB}_k$. Then $\nu_k\big(B(v,N_{j_i})\big) = B_k(v_k,N_{j_i})$. Denote by $U_n$ the path-connected component of $\Lambda_n \backslash B_n(v_n,N_{j_i})$ containing $\id_n$, and let $\mc{U}_n$ be a possibly empty set of remaining path-connected components. If $\bar{U}_n \in \mc{U}_n$, then $\bar{U}_n \subset \mc{TA}_n \cup \mc{TB}_n$, and for any $k \geq n$ the preimage $\bar{U}_k \owns v_k$ is a path-connected component of $\Lambda_k \backslash B_k(v_k,N_{j_i})$. Let $W_k = B_k(v_k,N_{j_i}) \cup \mc{U}_k$, then for each $k\geq n$ the complement $\Lambda_k \backslash W_k$ is path-connected. Set
   $$K_i = \lim_{\longleftarrow}\{W_k,\nu^{k+1}_k,k \geq n\}.$$

Let $s=(s_k),t=(t_k) \in L_v \backslash K_i$, and $\gamma$ and $\delta$ be paths joining $s$ and $t$ with $v$ respectively. Let $I \subset \mb{N}_0$ be a cofinal subsequence of indices such that if $k \in I$ then $v_k \in \mc{TA}_k \cup \mc{TB}_k$. Let $M = \max\{\ell(\gamma), \ell(\delta)\}$ and choose $n \in I$ such that $B_n(v_n,M) \subset \mc{TA}_n \cup \mc{TB}_n$. Then $\gamma_n = \nu_n \circ \gamma$ and $\delta_n = \nu_n \circ \delta$ are contained in $\mc{TA}_n \cup \mc{TB}_n$, and so either $s_n,t_n \in \mc{TA}_n \cup \{a^{2^n}\}$ or $s_n,t_n \in \mc{TB}_n \cup \{b^{2^n}\}$. Then there exists a path $c_n$ contained in one of these sets, which joins $s_n$ with $t_n$. From the construction of $\Lambda_k$ it follows that for any $k \in I$ such that $k>n$ the lift $c_k$ of $c_n$ such that one of its endpoints is $s_k$ has $t_k$ as another endpoint. So there is a path $c$ in $L_v \backslash K_i$ joining $s$ and $t$. It follows that $L_v$ has one end. 
\end{proof}

\subsubsection{The number of leaves with one or two ends}

It was shown in Blanc \cite{B} that if a minimal compact foliated metric space space has a residual set of points whose leaves have $2$ ends, then any leaf in $\mc{F}$ has one or two ends. In the Schori example the leaf $L_\id$ has four ends, therefore, the minimal foliation of $X_{\rm{Sch}}$ must have a residual set of points with leaves with one end. That is indeed the case, as Proposition \ref{Prop:meager} shows directly.

Denote by $\mc{AB} \subset V(\Lambda)$ the subset of vertices which lie in $4$- or $2$-ended path-connected components of $\Lambda$.

\begin{Prop}\label{Prop:meager}
The set $\mc{AB} \subset V(\Lambda)$ of vertices which belong to $4$- and $2$-ended path-connected components of $\Lambda$ is meagre.
\end{Prop}
\begin{proof} Consider the subset
  \begin{align*} A_0& = \big\{ (a^{\ell_k}) \in V(\Lambda) \ \big| \ ell_{k-1} = \ell_k \mod 2^{k-1}  \big\}. \end{align*}
All vertices in $A_0$ lie either in $L_\id$ or in a $2$-ended path-connected component of $\Lambda$, the first alternative occurring if $(a^{\ell_k})$ is eventually constant. Denote also
  \begin{align*} A_m & = \Big\{ (v_k) \in V(\Lambda) \ \big| \  \min_{(u_k) \in A_0} d \big( (v_k), (u_k) \big) \leq m \Big\}, \end{align*}
where $d$ denotes metric on the leaves of $\Lambda$. We are going to show that for each $m \geq 0$ the set $A_m$ is nowhere dense. For that it is enough to show that $A_m$ is closed, i.e. ${\rm Cl}(A_m) = A_m$, since $A_m$ has empty interior by minimality of $1$-ended path-connected components.

\begin{Lemma}\label{Lemma:closedsets}
The subset $A_m \subset V(\Lambda)$ is closed.
\end{Lemma}
\begin{proof} First let $m=0$. Let $\big\{(a^{\ell_k})_s \big\}$ be a sequence of elements in $A_0$ converging to a point $(\bar{a}_k) \in {\rm Cl}(A_0)$. Then for each $k \in \mb{N}$ there exists $N_k$ such that for all $s > N_k$ and all $i< k$ we have
  \begin{align*} a^{\ell_i}_s & = \bar{a}_i. \end{align*}
It follows that each $\bar{a}_k \in \nu_k(A_0)$ and so $(\bar{a}_k) \in A_0$.

Let $m>0$, and $\big\{(v_k)_s \big\} \in A_m$ be a sequence converging to $(\bar{v}_k) \in {\rm Cl}(A_m)$. Then
  \begin{align*}\min_{(u_k) \in A_0} d \big( (\bar{v}_k), (u_k) \big) & \leq m\end{align*}
since otherwise there exists an open ball around $\bar{v}_k$, for example, of radius $\breukje{1}{2}$, which does not contain points of $\big\{(v_k)_s \big\}$. Thus $(\bar{v}_k) \in A_m$. 
\end{proof}

Similarly, denote
  \begin{align*} B_0& = \big\{ (b^{\ell_k}) \in V(\Lambda) \ \big| \ b^{\ell_{k-1}} = b^{\ell_k} \mod 2^{k-1}  \big\} \end{align*}
and for $m > 0$ set
  \begin{align*} B_m & = \Big\{ (v_k) \in V(\Lambda) \ \big| \  \min_{(u_k) \in B_0} d \big( (v_k), (u_k) \big) \leq m \Big\}. \end{align*}
Then for $m \geq 0$ the set $B_m$ is nowhere dense, and $A_m \cup B_m$ is also nowhere dense. Any $2$-ended path-connected component intersects $A_0 \cup B_0$, so
   \begin{align*} \mc{AB} & = \bigcup_{m=0}^\infty \Big( A_m \cup B_m\Big),\end{align*}
and $\mc{AB}$ is a meagre set. 
\end{proof}

Proposition \ref{Prop:meager} implies that the set of vertices in $\Lambda$ which lie in $1$-ended path-connected components is residual, so the generic leaf in $X_{\rm Sch}$ has $1$ end \cite{CC1}.

\begin{Remark}
{\rm
The number of path-connected components of $\Lambda$ with $2$ ends is uncountable. Consider the dyadic solenoid
  \begin{align*} \Sigma & = \lim_{\longleftarrow} \{ \mb{S}^1, f^k_{k-1},\mb{N}_0 \}, \end{align*}
where $f^k_{k-1}:\mb{S}^1 \to \mb{S}^1$ is a $2$-fold covering. The
foliation of $\Sigma$ by path-connected components has an
uncountable infinity of leaves. Recall from Section
\ref{subsubsec:leavesSchori} that there is a subset of leaves with
$2$ ends such that each leaf contains a dyadic point $(a^{\ell_k})$,
and notice that to $(a^{\ell_k}) \in X_{\rm{Sch}}$ one can associate
a unique point in $\Sigma$. Using an argument on the existence of
paths as in Section \ref{subsubsec:leavesSchori} one sees that if
two points in $\Sigma$ represent different path-connected
components, then the corresponding points $(a^{\ell_i}) \in
X_{\rm{Sch}}$ represent different path-connected components of
$X_{\rm{Sch}}$. Therefore, $X_{\rm{Sch}}$ has an uncountable
infinity of leaves with $2$ ends. }
\end{Remark}

\subsection{Generalised Schori example: solenoid with a $2n$-ended path-connected component}\label{subsec:generalisedSchori}

The Schori example can be generalised to obtain a solenoid with a $2n$-ended path component, for any $n \in \mb{N}$. The construction is based on that of the original Schori example.

\subsubsection{The generalised Schori example with a $4n$-ended path-connected component}\label{subsubsec:schori4n}

Let $X_0$ be a genus $2$ surface with two $n$-handles $H_{1}$ and $H_{2}$ distinguished, let $C_1$ and $C_2$ be closed meridional curves in the interior of $H_1$ and $H_2$ respectively. The handles $H_1$ and $H_2$ intersect along their boundary curves. Cut the handles along $C_1$ and $C_2$, pull them apart by an appropriate homeomorphisms to obtain the cut surface $\widehat{X}_0$. Set $\bar{X}_0 = \rm{Cl}(\widehat{X}_0)$and consider copies $\bar{X}_0^{(i)}$, $i = 1,\ldots,2n+1$ of $\bar{X}_0$ with handles $\bar{H}_1^{(i)}$ and $\bar{H}_2^{(i)}$ and pairs of boundary circles ${C_1^{(i)}}'$ and ${C_1^{(i)}}''$, and ${C_2^{(i)}}'$ and ${C_2^{(i)}}''$ in each handle. Make identifications so that for $i = 1,\ldots,n$ ${C_1^{(2i-1)}}'$ is identified with ${C_1^{(2i)}}''$, ${C_1^{(2i-1)}}''$ is with ${C_1^{(2i)}}'$, ${C_2^{(2i)}}'$ with ${C_2^{(2i+1)}}''$, ${C_2^{(2i)}}''$ with ${C_2^{(2i+1)}}'$ (see Figure \ref{fig:first-identif} for $n=2$), and other cuts are identified trivially, i.e. ${C_t^{(s)}}'$ is identified with ${C_t^{(s)}}''$. Denote the obtained genus $2n+2$ surface by $X_1$. Distinguish $2n$ handles $H_1,\ldots,H_{2n}$ (those produced as a result of non-trivial identifications). Define $f^1_0:X_1 \to X_0$ as in the Schori example.

\begin{figure}
\centering
\includegraphics [width=7cm] {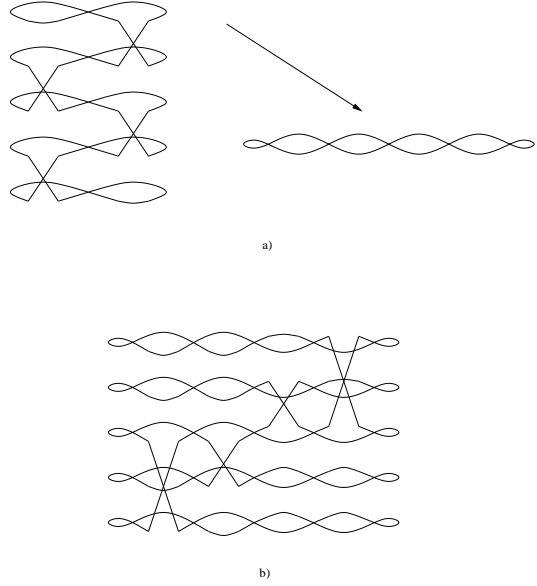}
\caption{Schreier diagrams for the weak solenoid with an $8$-ended path component: a) the $5$-fold covering space $\Lambda_1$ of the figure $8$, b) the $5$-fold covering space $\Lambda_2$ of $\Lambda_1$}
 \label{fig:first-identif}
\end{figure}

\begin{figure}
\centering
\includegraphics [width=3cm] {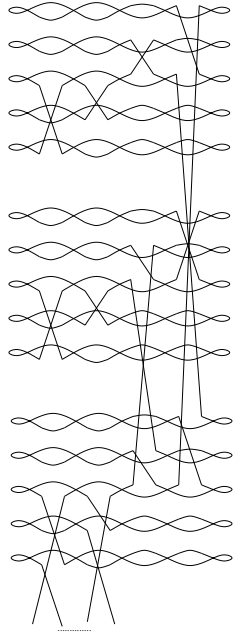}
\caption{Schreier diagrams for the weak solenoid with an $8$-ended path component: the $5$-fold covering space $\Lambda_3$ of $\Lambda_2$, three first layers.}
 \label{fig:second-identif}
\end{figure}

Let $X_k$ be a genus $\big((2n+1)^k - 2kn\big)$-surface, let $2n$ $2^k$-handles $H_1,\ldots,H_{2n}$ in $X_k$ be distinguished, and let $C_1,\ldots,C_{2n}$ be simple closed curves in the handles. Following the procedure described in the first step obtain $\bar{X}_k^{(i)}$, $i=1,\ldots,2n+1$ cut surfaces, so that in each of them the cut handles $\bar{H}_j^{(i)}$, $j=1,\ldots,2n$, are distinguished, and ${C^{(i)}_j}'$ and ${C^{(i)}_j}''$ are boundary curves contained in $\bar{H}^{(i)}_j$. Then make identifications as follows (see Figure \ref{fig:second-identif} for $n=2$):
\begin{enumerate}
\item for $1 \leq s \leq n$ identify ${C^{(s)}_{s}}'$ with ${C^{(n+1)}_{s}}''$, and ${C^{(s)}_{s}}''$ with ${C^{(n+1)}_{s}}'$,
\item for $n+2 \leq s \leq n+1$ identify ${C^{(s)}_{s-1}}'$ with ${C^{(n+1)}_{s-1}}''$, and ${C^{(s)}_{s-1}}''$ with ${C^{(n+1)}_{s-1}}'$,
\item identify other handles trivially.
\end{enumerate}

Denote the resulting identification space $X_{k+1}$, and define $f^{k+1}_k:X_{k+1} \to X_k$ as in the Schori example. Obtain $X_{\rm{Sch}}^{4n}$ as the inverse limit of the sequence of covering spaces $f^k_{k-1}$.

Choose a point $x_0 \in H_1 \cap H_2$. For each $k>0$ there is a unique point $x_k$ such that $f^k_0(x_k) = f^k_{k-1} \circ \cdots \circ f^1_0(x_k) = x_0$ and $x_0 \in H_n \cap H_{n+1}$, where $H_i$ are handles in $X_k$. Denote $x=(x_k)$. Using the method of the Schreier continuum one obtains that the path-connected component $L_x$ has $4n$ ends, topologically almost all leaves in $X_{\rm{Sch}}^{4n}$ have $1$ end and there is an uncountable infinity of path-connected components with $2$ ends.

\subsubsection{The generalised Schori example with a $4n+2$-ended path component}

Such an example is obtained by a slight alteration of the one in Section \ref{subsubsec:schori4n}: at the first step one identifies $2n$ cut copies of a genus $2$ surface $X_0$ along $2n-1$ cuts to obtain a genus $2n$ surface $X_1$. On the $k$-th step one considers a genus $\big((2n)^k - k(2n-1)\big)$-surface $X_k$ with $(2n-1)$ handles, which are identified as in Section \ref{subsubsec:schori4n}, only there is one less identification of the second type. The resulting space $X_{k+1}$ is a covering space of $X_0$.

So one obtains a weak solenoid $X_{\rm{Sch}}^{4n+2}$ with one $4n+2$-ended path component, which also has path-connected components with $1$ and $2$ ends. By the same argument as in Section \ref{subsubsec:schori4n} the generic path-connected component has $1$ end, and there is an uncountable infinity of path-connected components with $2$ ends.

\subsection{Rogers-Tollefson example revisited} \label{subsec:RTsolenoid}

The example introduced by Rogers and Tollefson \cite{RT} (see also Example \ref{Example:RT}) is that of a non-homogeneous solenoid
   $$K_\infty = \lim_{\longleftarrow} \{ K_i, f^i_{i-1},\mb{N}_0 \},$$
where $K_i$ is the Klein bottle and $f^i_{i-1}$ is a $2$-fold covering map constructed as follows (see Fokkink and Oversteegen \cite{FO}).

Represent the $2$-torus as the quotient space $\mb{R}/\mb{Z} \times \mb{R}/\mb{Z}$ and define
  \begin{align}\label{eq:barfcover} \bar{f}&:\mb{R}/\mb{Z} \times \mb{R}/\mb{Z} \to \mb{R}/\mb{Z} \times \mb{R}/\mb{Z}  : (x,y) \mapsto (x,2y), \end{align}
a $2$-fold covering map of the torus by itself. The Klein bottle is the quotient of $\mb{R}/\mb{Z} \times \mb{R}/\mb{Z}$ under the map $i:(x,y) \mapsto (x+ \breukje{1}{2},-y)$, and the covering projection \eqref{eq:barfcover} induces a $2$-fold covering map $f:K \to K$. Define the curves $b$ and $a$ in $K$ by $b(t) = (\breukje{1}{2}t,0)$ and $a(t) = (0,t)$ for $t \in [0,1]$. Then the homotopy classes $a = [a]$ and $b = [b]$ generate the fundamental group $\pi_1(K,\{0\})$ with a single relation $bab^{-1} = a^{-1}$. The induced map of the fundamental group is
  \begin{align*} f_\sharp & : \pi_1(K,\{0\}) \to \pi_1(K,\{0\}):(b,a) \mapsto (b,a^2), \end{align*}
with the corresponding decreasing chain of groups  $G_k = \langle b, a^{2k} \rangle$. Although $f$ is a regular covering map, the composition of $f$ with itself is not regular. The kernel of the group chain $\{G_k\}$ depends on the point, more precisely, for $\id = (\{0\})$ we have $\mc{K}_\id = \langle b \rangle$ and for any other point $x \in f_0^{-1}(\{0\})$ which does not belong to the path-connected component of $\id$ we have $K_x = \langle b^2 \rangle$. The Schreier diagrams for the pairs $(G_0,G_k)$ with $k=0,1,2,3$ are shown in Figure \ref{fig:RTsolenoid}.

\begin{figure}
\centering
\includegraphics [width=6cm] {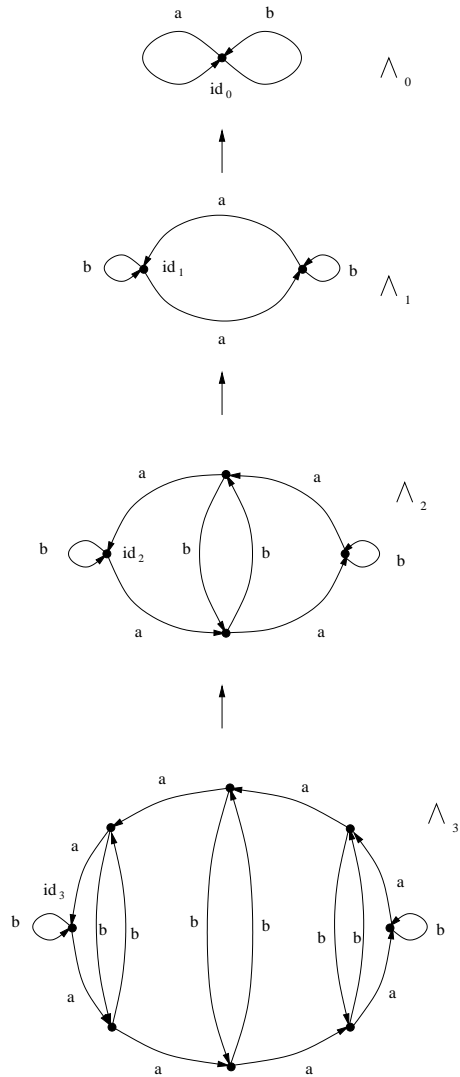}
\caption{Schreier diagrams for the Rogers-Tollefson solenoid for $i=0,1,2,3$}
 \label{fig:RTsolenoid}
\end{figure}

Using the argument on the lengths of paths similar to the ones in Section \ref{subsec:SchSchori} we obtain that the path-connected component of the point $\id \in K_\infty$ has $1$ end, and any other path-connected component in $K_\infty$ has $2$ ends.

\subsection{An example with non self-similar $G_0$-action}\label{subsec-nonselfsimilar}

We modify the Schori example so as to obtain a weak solenoid with non self-similar action of the fundamental group $G_0$ on the fibre. We first explain how self-similar actions arise in our context.

Recall the definition of a self-similar group action \cite{Nekr}. Let $S$ be a finite alphabet, and $S^*$ be the set of all finite words over $S$, including the empty word. The set $S^*$ can be thought of as a vertex set of a rooted tree (where the root is given by the empty word), where vertices $v$ and $v'$ are joined by an edge if and only if $v' = v s$ for some $s \in S$. If $\ell(v)$ is the length of the word $v \in S^*$, then $v$ and $v'$ are joined by an edge if and only if their lengths differ by $1$. Let $G_0$ be a group acting faithfully on $X^*$. Then the action of $G_0$ is \emph{self-similar} if and only if for every $s \in S$ and $g \in G$ there exist $t \in S$ and $h \in G$ (necessarily unique) such that
  $$g(sv) = t h(v) ~~ \textrm{for all} ~~ v \in S^*.$$
A nice property of the tree $X^*$ which reflects its self-similarity, is that there is a projection $xv \mapsto v$, which maps a subtree starting at vertex of length $1$ to the whole tree.

Now suppose a weak solenoid $X_\infty = \lim_{\longleftarrow}\{ f^n_{n-1}:X_n \to X_{n-1}, \mathbb{N}\}$ is given by a sequence of covering maps of constant degree $q$, and let $x_0 \in X_0$. Choose an alphabet $S$ on $q$ symbols. Then we can code preimages of $x_0$ in $X_n$ by words in $S^*$ of length $n$, thus associating to the solenoid a rooted tree. Then the fundamental group $G_0$ acts on $S^*$ in an obvious way, and one can come up with many examples where this action is self-similar. Using the associated tree, one can compute end structures of leaves by methods in \cite{BDN}, using the relation of self-similar actions to finite automatons. We now present an example which does not fit into the setting of self-similar actions, but where the number of ends of leaves can be computed by our method.

The most obvious way to construct such an example is to construct a weak solenoid as the inverse limit space of covering maps of variable degree, such that the situation cannot be reduced to the setting of a self-similar action of the fundamental group on the associated subtree by choosing a subsequence or getting rid of a finite number of factors in the sequence.
For that we will alternate $5$-to-$1$ and $3$-to-$1$ coverings of a genus $2$ surface as follows. Consider a sequence
  \begin{align*}5 ~ 3 ~ 3 ~ 5 ~ 3 ~ 3 ~ 3~ 5 ~3 ~3 ~3 ~3 ~5 ~3 \ldots\end{align*}
where the number $5$ occurs infinitely many times, and the number of occurences of $3$ between two occurences of $5$ grows by $1$ at each step. Each number in the sequence would correspond to the degree of a covering map $f^i_{i-1}: X_i \to X_{i-1}$, that is, $f^1_0: X_1 \to X_0$ is a $5$-fold cover, $f^2_1: X_2 \to X_1$ is a $3$-fold cover and so on (see Fig. \ref{fig:NSSsolenoid-identif}). The action of the fundamental group $G_0$ on the fibre of the corresponding solenoid $X_\infty = \lim_{\longleftarrow} \{f^i_{i-1}: X_i \to X_{i-1}\}$ is non-selfsimilar. The sequence of Schreier diagrams for this solenoid is presented in Fig. \ref{fig:NSSsolenoid}.

\begin{figure}
\centering
\includegraphics [width=4.5cm] {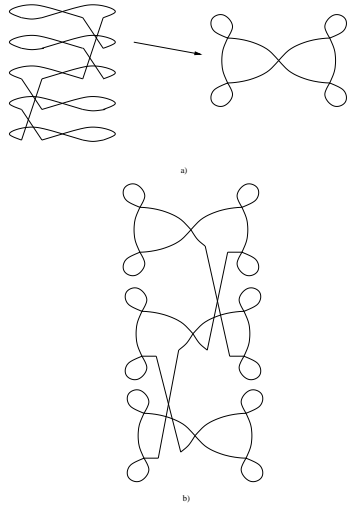}
\caption{Schreier diagrams for the non-selfsimilar solenoid for $i=0,1,2,3$}
 \label{fig:NSSsolenoid-identif}
\end{figure}

We use the method of the Schreier continuum to compute the number of ends of leaves in this solenoid.

\begin{Prop}
Let $X_\infty$ be a weak solenoid as above. Then leaves of $X_\infty$ have the following end structures.
\begin{enumerate}
\item There is a single leaf with $4$ ends.
\item There is a residual set of leaves with $1$ end.
\item There is an uncountable meager set of leaves with $2$ ends.
\end{enumerate}

\begin{figure}
\centering
\includegraphics [width=10.5cm] {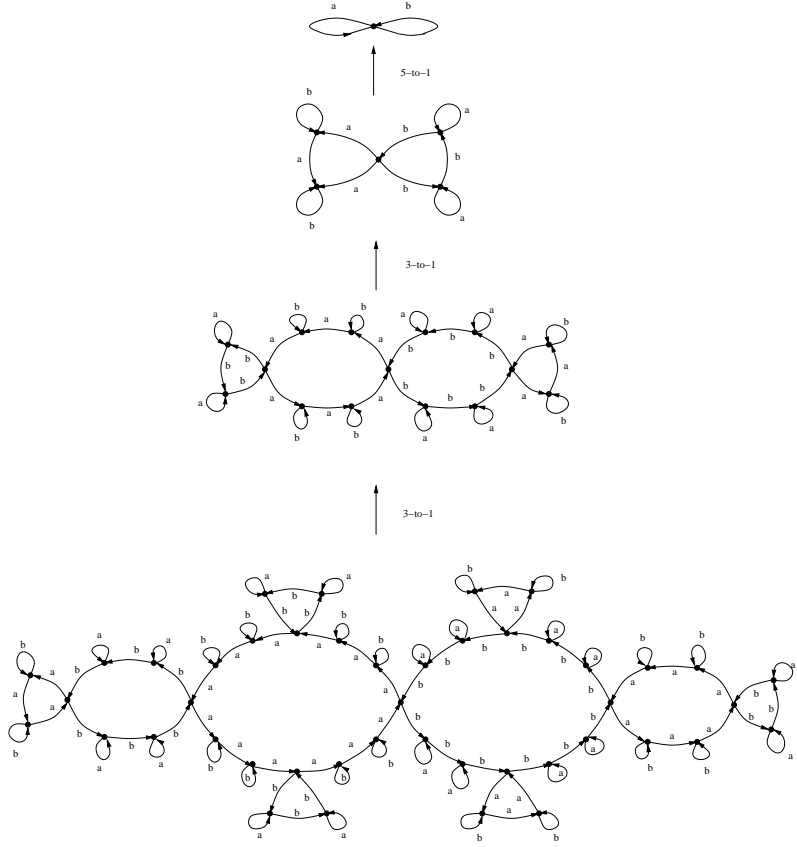}
\caption{Schreier diagrams for the non-selfsimilar solenoid for $i=0,1,2,3$}
 \label{fig:NSSsolenoid}
\end{figure}

\end{Prop}
\begin{proof} By a similar argument to the one in Proposition \ref{Prop:pathcomponentofid}, the path-connected component containing the point $(\id_n)$, where $\id_n$ denotes the coset of $G/G_n$ containing the identity, has $4$ ends. Let $q_0 = \id_0$, $q_1 = a$. For $k \in \mb{N}_0$ $q_{k-1}$ has a finite number of preimages under $f^k_{k-1}$. Let $q_k$ be the one at the maximal distance from $\id_k$. Then by argument similar to the one in Lemma \ref{Lemma:oneendcomponent} one shows that the path-connected component containing $(q_k)$ has $1$ end. Similarly, we can choose $(v_k)$ in such a way that $v_k= a^{\ell_k}$, and the distance between $v_k$ and $\id_k$ eventually grows. Then by an argument similar to that in Lemma \ref{Lemma:dyadicleaf} the path-connected component of $(v_k)$ has $2$ ends. As before, constructing a map from the set of leaves of a solenoid over a circle given by a suitable sequence of integers, one can show that the number of $2$-ended path-connected components in $X_\infty$ is uncountable.

We now have to show that every path-connected components in $X_\infty$ has $1$, $2$ or $4$ ends. Then it follows by the result of Blanc \cite{B} that generically leaves in $X_\infty$ have $1$ end. We use an argument similar to that in Section \ref{subsubsec:leavesSchoriother}. Basically, we need to compute the points whose projections are at variable lengths from $\id_k$, $k \in \mb{N}_0$.

Let $\nu^{k+1}_k$ be a $3$-fold cover. As in Section \ref{subsubsec:leavesSchori}, denote by $\mc{A}_k$ the path-connected component of $\Lambda_k \backslash \{\id_k\}$ containing cosets $a^i=[a^i] \in G/G_k$, and by $\mc{B}_k$ the other one. The vertex $\id_k$ has three preimages under $\nu^{k+1}_k: \Lambda_{k+1} \to \Lambda_k$, those being the vertices $\id_{k+1}, a^{2^k}, b^{2^k}$. Then $\mc{A}_{k+1} \backslash \{a^{2^k}\}$ consists of three path-connected components: the component $\mc{A}^+_{k+1}$ containing vertices $a^i$ for $i<2^k$, the component $\mc{A}^-_{k+1}$ containing vertices $a^i$ for $ i > 2^k$, and the remaining component $\mc{TA}_{k+1}$. Similarly, $\mc{B}_{k+1} \backslash  \{b^{2^k}\}$ consists of path-connected components $\mc{B}^+_{k+1}$, $\mc{B}^-_{k+1}$ and $\mc{TB}_{k+1}$.

Let $\nu^{k+1}_k$ be a $5$-fold cover. As before, denote by $\mc{A}_k$ the path-connected component of $\Lambda_k \backslash \{\id_k\}$ containing cosets $a^i=[a^i] \in G/G_k$, and by $\mc{B}_k$ the other one. The vertex $\id_k$ has five preimages under $\nu^{k+1}_k: \Lambda_{k+1} \to \Lambda_k$, those being the vertices $\id_{k+1}$ and cosets $a^m, a^{2m}, b^m, b^{2m}$. Then $\mc{A}_{k+1} \backslash \{a^{m},a^{2m}\}$ consists of five path-connected components: the components $\mc{A}^j_{k+1}$, $j=1,2,3$, such that $\mc{A}_{k+1}^j$ contains vertices $a^i$ for $(j-1)m<i<jm$, and the components $\mc{TA}^{s}_{k+1}$, $s=1,2$, such that $a^{sm}$ is in the closure of $\mc{TA}^{s}_{k+1}$. Similarly, $\mc{B}_{k+1} \backslash  \{b^m,b^{2m}\}$ consists of path-connected components $\mc{B}^j_{k+1}$, $j=1,2,3$, and $\mc{TB}_{k+1}^s$, $s=1,2$.

Let $v = (v_k) \in \Lambda$ be a vertex, and assume that $v \notin L_\id$. If for some $k$ we have $v_k \in \mc{A}_k$ (resp. $v_k \in \mc{B}_k$), then if $\nu^{k+1}_k$ is $3$-fold, either $v_{k+1} \in \mc{A}_{k+1}^+ \cup \mc{A}^-_{k+1}$ (resp. $v_{k+1} \in \mc{B}_{k+1}^+ \cup \mc{B}^-_{k+1}$) or $v_{k+1} \in \mc{TB}_{k+1}$ (resp. $v_{k+1} \in \mc{TA}_{k+1}$). If $\nu^{k+1}_k$ is $5$-fold, either $v_{k+1} \in \mc{A}_{k+1}^j$, $j=1,2,3$ (resp. $v_{k+1} \in \mc{B}_{k+1}^j$, $j=1,2,3$), or $v_{k+1} \in \mc{TB}_{k+1}^s$, $s=1,2$, (resp. $v_{k+1} \in \mc{TA}_{k+1}^s$, $s=1,2$).

Then the following situations are possible:
  \begin{enumerate}
  \item $v_k \in \mc{A}_{k}^+ \cup \mc{A}^-_{k} \cup \mc{B}_{k}^+ \cup \mc{B}^-_{k}$ if $\nu^k_{k-1}$ is $3$-fold and otherwise $v_k \in \mc{A}_k^1 \cup \mc{A}_k^2 \cup \mc{A}_k^3 \cup \mc{B}_k^1 \cup \mc{B}_k^2 \cup \mc{B}_k^3$ for at most a finite number of $k$'s, then argument similar to that in Lemma \ref{Lemma:oneendcomponent} shows that $L_v$ has $1$ end.
  \item $v_k \in \mc{TA}_{k} \cup \mc{TB}_{k}$, if $\nu^k_{k-1}$ is $3$-fold, and otherwise $v_k \in \mc{TA}_{k}^1 \cup \mc{TA}_{k}^2 \cup \mc{TB}_{k}^1 \cup \mc{TB}_{k}^2$ for at most a finite number of indices, then $v$ is called a \textit{dyadic} point, and by an argument similar to that in Lemma \ref{Lemma:dyadicleaf} $L_v$ has $2$ ends.
  \item there is a cofinal subset $I \in \mb{N}_0$ such that for any $k \in I'$ we have $v_k \in \mc{TA}_k \cup \mc{TB}_k$ if $\nu^k_{k-1}$ is $3$-fold, and otherwise $v_k \in \mc{TA}_{k}^1 \cup \mc{TA}_{k}^2 \cup \mc{TB}_{k}^1 \cup \mc{TB}_{k}^2$, and the subset $\mb{N}_0 \backslash I$ is also cofinal. In this case $v$ is called a \textit{flip-flopping} point. Then by an argument similar to that in Lemma \ref{Lemma:genericleaf} $L_v$ has $1$ end.
  \end{enumerate}
\end{proof}

\section{Conclusions}\label{sec:conclusions}

For a foliated bundle $p: E \to B$ with foliation $\mc{F}$ we have
introduced the notion of the \emph{Schreier continuum} $\Lambda$
which can be thought of as the union of holonomy graphs of leaves of
the minimal set of $\mc{F}$ with suitable topology, and have shown
that in the case when the minimal set of $\mc{F}$ is transversely a
Cantor set, a great deal can be said about asymptotic properties of
leaves in the minimal set by means of the study of the associated
inverse limit representation of $\Lambda$.

In particular, using the method of the Schreier continuum we have shown that, given $n>1$, based on the construction of Schori \cite{S}, one can construct a solenoid where a single leaf has $2n$ ends, there is an uncountable infinity of leaves with $2$ ends which form a meager subset, and an uncountable infinity of leaves with $1$ end, which form a residual subset. We have also shown that the method of Schreier continuum can be used to analyse quite complicated examples, for example, solenoids with non self-similar action of the fundamental group on the fibre.

Blanc \cite{B} proves that if a residual set of points in a foliated space $E$ has leaves with $2$ ends, then almost every leaf in $E$ (with respect to a harmonic measure) has $2$ ends. Blanc also announces an example showing that this need not be the case if $E$ has a residual set of points with leaves with $1$ end. It would be of interest to determine whether this can occur in solenoids.

\begin{Q}
For the structure of ends in a solenoid, does topologically ``almost
all" imply measure-theoretically ``almost all" ?
\end{Q}

\bibliographystyle{123}

{\small \hfill Received September 28, 2011}

{\small \hfill Revised version received January 6, 2012}
\end{document}